\documentclass[10pt]{article}

\usepackage{setspace}

\usepackage{amsmath,amssymb,amsthm}
\usepackage{epsfig}
\usepackage{psfrag}
\begin{document}

\numberwithin{equation}{section}
\newtheorem{theorem}{Theorem}
\newtheorem{proposition}[theorem]{Proposition}
\newtheorem{remark}[theorem]{Remark}
\newtheorem{lemma}[theorem]{Lemma}
\newtheorem{corollary}[theorem]{Corollary}
\newtheorem{definition}[theorem]{Definition}

\newcommand{\RR}{\mathbb{R}}
\newcommand{\NN}{\mathbb{N}}
\newcommand{\ZZ}{\mathbb{Z}}
\newcommand{\PP}{\mathbb{P}}
\newcommand{\EE}{\mathbb{E}}
\newcommand{\Var}{\mathbb{V}{\rm ar}}
\newcommand{\Cov}{\mathbb{C}{\rm ov}}
\newcommand{\Id}{\text{Id}}
\newcommand{\dive}{\operatorname{div}}
\newcommand{\dist}{\operatorname{dist}}
\newcommand{\Tr}{\operatorname{Tr}}

\newcommand{\dps}{\displaystyle}

\def\longrightharpoonup{\relbar\joinrel\rightharpoonup}

\title{Asymptotic behaviour of Green functions of divergence form operators  
with periodic coefficients}
\author{A. Anantharaman$^{1,4}$, X. Blanc$^{2}$ and F. Legoll$^{3,4}$\\
{\footnotesize $^1$ CERMICS, \'Ecole Nationale des Ponts et
Chauss\'ees, Universit\'e Paris-Est,}\\
{\footnotesize 6 et 8 avenue Blaise Pascal, 77455 Marne-La-Vall\'ee
Cedex 2, France}\\
{\footnotesize \tt ananthaa@cermics.enpc.fr}\\
{\footnotesize $^2$ CEA, DAM, DIF,
91297 Arpajon, France}\\
{\footnotesize \tt blanc@ann.jussieu.fr, Xavier.Blanc@cea.fr}\\
{\footnotesize $^3$ Institut Navier, LAMI, \'Ecole Nationale des Ponts et
Chauss\'ees, Universit\'e Paris-Est,}\\
{\footnotesize 6 et 8 avenue Blaise Pascal, 77455 Marne-La-Vall\'ee
Cedex 2, France}\\
{\footnotesize \tt legoll@lami.enpc.fr}\\
{\footnotesize $^4$ INRIA Rocquencourt, MICMAC team-project,}\\
 {\footnotesize Domaine de Voluceau, B.P. 105,
 78153 Le Chesnay Cedex, France}
}
\date{\today}

\maketitle

\begin{abstract}
This article is concerned with the asymptotic behaviour, at infinity and
at the origin, of Green functions of operators of the form $Lu = -\dive\left(A
  \nabla u\right),$ where $A$ is a periodic, coercive and bounded matrix. 
\end{abstract}

\section{Introduction}

The study of Green's functions for elliptic operators is an important
research subject. It is linked with many different fields, as for
instance homogenization
\cite{avellaneda-lin-87,avellaneda-lin-91,bonnetier,mprf,rate}, or 
the study of singular points~\cite{gruter,littmann}. 
The aim of the present article is to provide explicit bounds at infinity
for the Green function $G$ of a divergence-type elliptic operator with
periodic coefficients. 
Many arguments in this paper are already present in the literature in a
scattered manner, and our main contribution is to put them together in a
clear way. 
Our arguments also provide us with explicit bounds on $G$ in
the neighbourhood of the origin, where $G$ is singular. These latter
results are already described in a comprehensive way in the
literature. 

In all the article, we assume that $d\geq 2$ is the dimension of the
ambient space, and that (here, $\RR^{d\times d}$ is the space of square
matrices of size $d$) the field $A:\RR^d \longrightarrow \RR^{d\times
  d}$ satisfies
\begin{gather}
  \label{eq:A-periodique}
A \text{ is } \ZZ^d \text{ periodic, } 
\\
\label{eq:A-holder}
  A \text{ is } \delta - \text{H\"older continuous for some $\delta >0$,} 
\\
  \label{eq:A-coercive}
  \exists\alpha >0, \quad \forall \xi \in \RR^d, \quad \forall x\in
  \RR^d, \quad \xi^T A(x) \xi \geq \alpha |\xi |^2,
\end{gather}
where $| \cdot |$ is the Euclidean norm of $\RR^d,$ and
\begin{equation}
  \label{eq:A-bornee}
  A \in L^\infty\left(\RR^d, \RR^{d\times d}\right).
\end{equation}

We want to study the behaviour at infinity of the
Green function $G$ associated with the operator 
$$
L = -\dive\left(A \nabla\cdot\right),
$$
that is, the function $G : \RR^d \times \RR^d \longrightarrow \RR$
such that 
\begin{equation}
\label{eq:def-G}
-\dive_x\left( A(x) \nabla_x G(x,y) \right) = \delta_y(x).
\end{equation}
See~\eqref{eq:def-G-faible} below for a more precise formulation.
By behaviour at infinity, we mean the asymptotic of $G(x,y)$ as $|x-y|$
goes to infinity. 
This question has been widely studied in the literature. According
to~\cite[Theorem~13]{avellaneda-lin-87} (see also~\cite{kozlov}), we
have, if $d\geq 3$, 
\begin{equation}
\label{eq:decroissance-green-0}
\exists C, \quad \forall (x,y)\in \RR^d\times \RR^d, \quad |G(x,y)|\leq 
C |x-y|^{2-d}.
\end{equation}
In addition (see~\cite[Theorem 13]{avellaneda-lin-87}), we have, in the case $d=2$,
\begin{equation}
\label{eq:decroissance-green-0-2D}
\exists C, \quad \forall (x,y)\in \RR^2\times \RR^2, \quad |G(x,y)|\leq 
C \left(1+\log|x-y|\right).
\end{equation}
Note that these estimates characterize both the asymptotic behaviour of
$G$ at infinity (when $|x-y| \to \infty$) and at the origin (when
$|x-y| \to 0$). 
An important point here is that many papers consider only the case of
Green functions for operators $L$ defined in a {\em bounded} domain (the
equation~\eqref{eq:def-G} is then complemented by appropriate boundary
conditions). This is 
the case for instance of~\cite{dolzmann} and~\cite[Theorems~1.1
and~3.3]{gruter}. This is also the case
of~\cite[Theorem~13]{avellaneda-lin-87}, although a remark following
this Theorem indicates that the constant in the estimate can be chosen
independent of the domain. In~\cite[Theorem~3.3]{gruter}, bounds are
provided on $G$, its gradient and the second derivatives $\nabla_x
\nabla_y G$, in the case $d 
\geq 3$. A remark following that result points out that the constant in the
estimate of $G$ is independent of the domain, whereas the constants in
the estimates of the derivatives of $G$ a priori depend on the domain.

In this article, we also address the question of the decay of the
{\em derivatives} of $G$ at infinity. We have, as proved in
Propositions~\ref{pr:estimee-gradient}
and~\ref{pr:estimee-gradient-croise} below (the material is present
in~\cite{avellaneda-lin-87}, and also in~\cite{bonnetier}), for any
$d\geq 2$,
\begin{equation}
\label{eq:decroissance-green-1}
\exists C>0, \ \ \forall (x,y)\in \RR^d\times \RR^d, \ \
|\nabla_x G(x,y)| + |\nabla_y G(x,y)| \leq  C|x-y|^{1-d}
\end{equation}
and
\begin{equation}
\label{eq:decroissance-green-2}
\exists C>0, \quad \forall (x,y)\in \RR^d\times \RR^d, \quad
|\nabla_x\nabla_y G(x,y)| \leq  C|x-y|^{-d}.
\end{equation}

\medskip

A preliminary question, before
showing~\eqref{eq:decroissance-green-0},~\eqref{eq:decroissance-green-0-2D},~\eqref{eq:decroissance-green-1}
and~\eqref{eq:decroissance-green-2}, is the existence and uniqueness of
$G$ defined 
by~\eqref{eq:def-G}. This question is addressed in~\cite[Theorem
1.1]{gruter}, for the Green function in a bounded domain $\Omega \subset
\RR^d$ with homogeneous Dirichlet boundary conditions. An existence proof
is then provided for $G$ such that $\nabla_x G(\cdot,y) \in L^p(\Omega
\setminus B_r(y))$, for any $p>d/(d-1)$ and $r>0$. Actually, in~\cite{gruter},
only the case $d\geq 3$ is studied, but the existence proof carries
through to the case $d=2$. 
The uniqueness of $G$, under the assumption that $G\geq 0$, is also
proved in~\cite[Theorem 1.1]{gruter} for $d\geq 3$. The case $d=2$ is
not covered by their proof. 
A proof of uniqueness when $d=2$
can be found in the appendix of~\cite{kenig-ni}, both for a bounded
domain and for the whole space.

We finally mention that the case of non-divergence form operators (of
parabolic and elliptic type) has also been considered, see
e.g.~\cite{escaurazia}.

\medskip

The article is organized as follows. In Section~\ref{sec:definition}, we
give existence and uniqueness theorems for Green functions. In
Section~\ref{sec:asymptotique}, we state asymptotic properties on $G$
and its derivatives. Finally,
we give in Section~\ref{sec:extensions} some remarks about possible
extensions of the results stated in the present article.

\section{Definition of Green function}
\label{sec:definition}

In order to state the existence and uniqueness result for $G$ solution
of~\eqref{eq:def-G}, we first write a weak formulation: we look for $G :
\RR^d \times \RR^d \mapsto \RR$ such that
\begin{equation}
\label{eq:def-G-faible}
\forall y \in \RR^d, \quad \forall \varphi \in {\cal D}(\RR^d), \quad
\int_{\RR^d} \left( \nabla \varphi(x) \right)^T A(x) \nabla_x G(x,y)\, dx =
\varphi(y).
\end{equation}
In the sequel, we will need the definition of weak $L^p$ spaces,
which are special cases of Lorentz spaces: for any open subset $\Omega
\subset \RR^d$, for any $p \in [1,\infty]$, 
$$
L^{p,\infty}(\Omega) = \left\{f:\Omega \to \RR, \ f \text{ measurable,} \
\|f\|_{L^{p,\infty}(\Omega)} < \infty \right\},
$$
where
$$
\| f \|_{L^{p,\infty}(\Omega)} =  \sup_{t\geq 0} \left\{t \, \mu 
\left( \{ x \in \Omega, \ |f(x)|\geq t\} \right)^{1/p} \right\},
$$
where $\mu$ is the Lebesgue measure. We recall (see
e.g.~\cite[p.~8]{bergh-lofstrom} or~\cite{bonnetier}) that, for
any $0 < \beta < p-1$,
\begin{equation}
\label{eq:injection-lorentz}
C(p,\beta,\Omega) \, \| f \|_{L^{p-\beta}(\Omega)}
\leq
\| f \|_{L^{p,\infty}(\Omega)}
\leq 
\| f \|_{L^p(\Omega)},
\end{equation}
with $\dps C(p,\beta,\Omega) = 
\left( \frac{p}{\beta} \right)^{1/(p-\beta)} \
\left( \mu(\Omega) \right)^{\frac{-\beta}{p(p-\beta)}}$.

\begin{theorem}[Existence and uniqueness of $G$, $d\geq 3$]
\label{th:existence-unicite}
Let $d\geq 3$, and assume that $A$ satisfies~\eqref{eq:A-coercive}
and~\eqref{eq:A-bornee}. Then, equation~\eqref{eq:def-G-faible} has a
unique solution in $L^\infty_y\left(\RR^d,W^{1,1}_{x,\rm
    loc}(\RR^d)\right)$ such that  
\begin{equation}
\label{eq:G-lim}
\lim_{|x-y|\to \infty} G(x,y) =0.
\end{equation}
Moreover, $G$ satisfies the following estimate:
\begin{equation}
  \label{eq:estimee-G}
\forall q < \frac d {d-1}, \quad  \forall y \in \RR^d, \quad G(\cdot, y)
\in W_{\rm loc}^{1,q}\left(\RR^d\right) \cap W^{1,2}_{\rm
  loc}\left(\RR^d\setminus \{y\} \right)
\end{equation}
and
\begin{equation}
\label{eq:G-positive}
\exists C, \quad \forall (x,y)\in \RR^d\times \RR^d, \quad 
0 \leq G(x,y) \leq C |x-y|^{2-d}.
\end{equation}
\end{theorem}

\begin{proof}
First, note that, according to~\cite[Theorem 8.24]{gilbarg-trudinger},
the function $G$ is H\"older continuous with respect to $x$ and $y$
whenever $x\neq y$. The same property holds for $G_R$ defined below.

Let $R>0$. We first define $G_R$ as the Green
function of the operator $-\dive\left(A\nabla \cdot\right)$ on the ball
$B_R = B_R(0)$ with homogeneous Dirichlet boundary conditions, that is, 
\begin{equation}
\label{eq:def-G-faible-R}
\forall y \in B_R, \quad \forall \varphi \in {\cal D}(B_R), \quad
\int_{B_R} \left( \nabla \varphi(x) \right)^T A(x) \nabla_x G_R(x,y) dx =
\varphi(y),
\end{equation}
and $G_R(x,y) = 0$ if $|x|=R$. Applying~\cite[Theorem 1.1]{gruter}, we
know that such a $G_R$ exists, and satisfies
\begin{gather}
\label{eq:borne-G-1}
\forall y \in B_R, \quad \|G_R(\cdot,y)\|_{L^{\frac d {d-2},\infty}(B_R)}
\leq C,
\\
\label{eq:borne-G-2}
\forall y \in B_R, \quad \|\nabla_x G_R(\cdot,y)\|_{L^{\frac d {d-1},\infty}(B_R)}
\leq C,
\end{gather}
and
\begin{equation}
\label{eq:borne-G-3}
\forall (x,y) \in B_R\times B_R, \quad 
0 \leq G_R(x,y) \leq \frac{C}{|x-y|^{d-2}},
\end{equation}
where $C>0$ does not depend on $R$ and $y$. 

Next, we note that if $R'>R$, then, due to the maximum principle, we
have $G_{R'} \geq G_R$ in $B_R\times B_R$. Thus, $G_R$ is a
non-decreasing function of $R$. With the help of~\eqref{eq:borne-G-3},
this implies that the function $G_R$
converges almost everywhere to some function $G$, defined on $\RR^d
\times \RR^d$, and that satisfies~\eqref{eq:G-positive}. 
This implies~\eqref{eq:G-lim}. 
In addition, we deduce from~\eqref{eq:borne-G-3} that $G_R$ converges to
$G$ in $L^p_{\rm loc}(\RR^d \times \RR^d)$, for any
$p<d/(d-2)$, and that, for any $y \in \RR^d$, the function $G_R(\cdot,y)$
converges to $G(\cdot,y)$ in $L^p_{\rm loc}(\RR^d)$, for any
$p<d/(d-2)$.  

In view of~\eqref{eq:borne-G-2} and~\eqref{eq:injection-lorentz}, we see
that, for any bounded domain $\Omega \subset \RR^d$, and for any $q < d/(d-1)$, there
exists $C(\Omega,q,d)$ such that
$$ 
\forall R \text{ s.t. $\Omega \subset B_R$},  \quad \forall y \in B_R, \quad
\left\| \nabla_x G_R(\cdot,y) \right\|_{L^q(\Omega)} \leq C(\Omega,q,d).
$$
Hence, extracting a subsequence if necessary, $\nabla_x G_R(\cdot,y)$
converges weakly in $\left( L^q(\Omega) \right)^d$ to some $T \in
\left( L^q(\Omega) \right)^d$. Recall now that $G_R(\cdot,y)$ converges
to $G(\cdot,y)$ in $L^p_{\rm loc}(\RR^d)$, for any $p<d/(d-2)$. Hence $T
= \nabla_x G_{|\Omega}$, and $\nabla_x G_R(\cdot,y)$ converges to $\nabla_x G$
weakly in $\left( L^q(\Omega) \right)^d$, for any bounded domain
$\Omega$ and any $q < d/(d-1)$. 
Passing to the limit in~\eqref{eq:def-G-faible-R}, we see that $G$ is a
solution to~\eqref{eq:def-G-faible}.

Finally, the bounds~\eqref{eq:borne-G-1} and~\eqref{eq:borne-G-2} imply,
together with~\eqref{eq:injection-lorentz},
that $G\in L^\infty_y\left(\RR^d, W_{x,\rm loc}^{1,1}(\RR^d)\right)$. We
have thus proved the existence of $G$. 

\smallskip

Property~\eqref{eq:estimee-G} is proved in~\cite[Theorem 1.1]{gruter},
and its proof does not depend on the fact that the domain used there is
bounded. Note that we have already proved part of this
property. Indeed, as pointed above, 
for any $y \in \RR^d$, we have $G(\cdot,y) \in L^p_{\rm loc}(\RR^d)$ for any
$p<d/(d-2)$ and $\nabla_x G(\cdot,y) \in \left( L^q_{\rm loc}(\RR^d)
\right)^d$ for any $q<d/(d-1)$, thus $G(\cdot,y) \in W^{1,q}_{\rm
  loc}(\RR^d)$ for any $q<d/(d-1)$.

\medskip

In order to prove uniqueness, we assume that $G_1$ and $G_2$ are
two solutions, and point out that $H = G_1 - G_2$ satisfies $\dive_x
\left(A\nabla_x H\right) = 0$ for any $y\in \RR^d$. Fixing $y$, we apply
the corollary of~\cite[Theorem 4]{moser}, 
which implies that, if $H$ is not constant, then 
$\sup\left\{H(x,y), |x-y|=r \right\} - \inf\left\{H(x,y),
  |x-y|=r \right\}$ must grow at least like a positive power of $r$ as $r\to
\infty$. This latter behaviour is in contradiction
with~\eqref{eq:G-lim}. Thus $H=G_1-G_2$ is constant, and~\eqref{eq:G-lim}
implies that $G_1 \equiv G_2$.

Note finally that the corollary of~\cite[Theorem 4]{moser} is stated in
the case when $A$ 
is symmetric, but the same result holds in the non-symmetric
case. Indeed, Harnack's inequality is still valid in such a case,
see
e.g.~\cite[Theorem 8.20]{gilbarg-trudinger},~\cite[Theorem
5.3.2]{morrey} or~\cite{kenig-koch}. 
\end{proof}

\begin{theorem}[Existence and uniqueness of $G$, $d=2$]
\label{th:existence-unicite-2d}
Let $d=2$, and assume that $A$ satisfies~\eqref{eq:A-coercive}
and~\eqref{eq:A-bornee}. Then, equation~\eqref{eq:def-G-faible} has a
unique (up to the addition of a constant) 
solution in $L^\infty_{y,\rm loc}\left(\RR^d, W^{1,1}_{x,\rm
    loc}(\RR^d)\right)$ such that 
\begin{equation}
\label{eq:G-positive-2d}
\exists C>0, \quad \forall (x,y) \in \RR^d\times \RR^d, \quad |G(x,y)|
\leq C\left( 1 + \bigl|\log|x-y|\bigr| \right).
\end{equation}
Moreover, $G$ satisfies the following estimate:
\begin{equation}
  \label{eq:estimee-G-2d}
\forall q < 2, \quad  \forall y \in \RR^d, \quad G(\cdot, y)
\in W_{\rm loc}^{1,q}\left(\RR^d\right) \cap W^{1,2}_{\rm
  loc}\left(\RR^d\setminus \{y\} \right).
\end{equation}
\end{theorem}

\begin{proof}
The proof of this result may be found in the appendix
of~\cite{kenig-ni}. However, for the sake of completeness, we provide an
alternative proof. This proof, in contrast to that of~\cite{kenig-ni},
relies on basic tools of analysis of PDEs. 

\medskip
We use the same strategy as in the proof of
Theorem~\ref{th:existence-unicite}, defining first the Green function
$G_R$ of the operator $L$ on $B_R$. However, we cannot simply apply the
results of~\cite{gruter} to define $G_R$, as those results are stated in
dimension $d \geq 3$. It is 
possible to adapt the proof of~\cite[Theorem 1.1]{gruter} to the
two-dimensional case, but a simpler proof consists in following the
approach of Section~6 
of~\cite{dolzmann}. These results give the 
existence and uniqueness of $G_R$ solution to~\eqref{eq:def-G-faible-R} 
in the ball $B_R = B_R(0) \subset \RR^2$,
with the homogeneous Dirichlet boundary conditions $G_R(x,y) = 0$ if
$|x|=R$, in $W^{1,p}(B_R)$ for any $p<2$. In addition, it is shown
in~\cite[Section 6]{dolzmann} that estimate~\eqref{eq:borne-G-2} holds, 
namely 
\begin{equation}
\label{eq:bibi2}
\forall y \in B_R, \quad \|\nabla_x G_R(\cdot,y)\|_{L^{2,\infty}(B_R)} \leq C
\end{equation}
for a constant $C$ independent of $R$ and $y$. 

\medskip

\noindent{\bf Step 1: passing to the limit $R \to \infty$ on $G_R$}

\medskip

Consider the domain $\Omega = B_{R'}$, with $R'$ fixed, and consider
next $R > R'$.  
Applying~\eqref{eq:injection-lorentz}
to $\nabla_x G_R(\cdot,y)$ on $\Omega$, we see that~\eqref{eq:bibi2}
implies that $\nabla_x G_R$ is
bounded in $\left( L^q(B_{R'} \times B_{R'}) \right)^2$ for any $q < 2$,
independently of $R$. Hence, extracting a subsequence if necessary,
$\nabla_x G_R$ 
converges weakly in $(L^q(B_{R'} \times B_{R'}))^2$ to $T \in
(L^q(B_{R'} \times B_{R'}))^2$. Now, we have, in the sense of
distribution,
$$
\partial_{x_1} \partial_{x_2} G_R = \partial_{x_2}\partial_{x_1}
G_R.
$$
This property passes to the limit, so that $\partial_{x_1} T_2
= \partial_{x_2} T_1$. This implies that $T = \nabla_x G$ for some
$G \in W^{1,q}\left(B_{R'} \times B_{R'}\right)$. Next, we point out
that this limit does not depend on $R'$ in the sense that if $R'' > R'$,
then $\nabla_x G'$ obtained in $B_{R'}$ is equal to $\nabla_x
G''_{|B_{R'}},$ where $\nabla_x G''$ is obtained in $B_{R''}$. 
Hence $G \in W^{1,q}_{\rm loc}\left(\RR^2\times
\RR^2\right) \subset L^q_{y,\rm loc}\left(\RR^2, W^{1,1}_{x,\rm
  loc}(\RR^2)\right)$. 
Passing to the limit in~\eqref{eq:def-G-faible-R}, we obtain that $G$ is
a solution to~\eqref{eq:def-G-faible}. 
Until now, the function $G(\cdot,y)$ is only
determined up to a constant. We fix this constant by choosing
$G(\cdot,y)$ such that 
\begin{equation}
\label{eq:choice-0}
\int_{B_1(y)} G(x,y) \, dx = 0.
\end{equation}

\medskip 

To prove the existence of a function $G$ satisfying the claimed
properties, we are now left with showing that the function $G$ that we
have built satisfies~\eqref{eq:G-positive-2d}
and~\eqref{eq:estimee-G-2d}. 

\medskip 

\noindent{\bf Step2: proving that $G$ satisfies~\eqref{eq:estimee-G-2d}}

\medskip

By construction, we have 
$G(\cdot,y) \in W^{1,q}(\Omega)$, for any $q<2$ and any bounded domain
$\Omega$. The proof of the fact that $G(\cdot,y) \in W^{1,2}_{\rm loc}(\RR^2
  \setminus y)$ follows the same lines as the proof given
  in~\cite[Theorem 1.1]{gruter}, which does not depend on the fact that
  the domain used there is bounded, nor on the fact that the dimension
  there is $d \geq 3$. We thus have proved~\eqref{eq:estimee-G-2d}.

\medskip

\noindent{\bf Step 3: proving that $G$ satisfies~\eqref{eq:G-positive-2d}}

\medskip

We first infer
from~\eqref{eq:bibi2} and~\eqref{eq:injection-lorentz} that, for any
bounded domain $\Omega \subset B_R$ and any $y \in B_R$, we have
$$
\frac{2}{\sqrt{\mu(\Omega)}} \
\|\nabla_x G_R(\cdot,y)\|_{L^1(\Omega)} \leq C
$$
for a constant $C$ independent of $R$, $\Omega$ and $y$. Since $\nabla_x
G_R(\cdot,y)$ weakly converges to $\nabla_x G(\cdot,y)$, we deduce that
\begin{equation}
\label{eq:estimee-gradient-G-1}
\frac{2}{\sqrt{\mu(\Omega)}} \
\|\nabla_x G(\cdot,y)\|_{L^1(\Omega)} \leq C
\end{equation}
for a constant $C$ independent of $\Omega$ and $y$. Note that this
implies that $G \in L^\infty_{y,\rm loc} \left(\RR^d, W^{1,1}_{x,\rm
    loc}(\RR^d)\right)$, as claimed in the theorem.

Second, we apply Poincar\'e-Wirtinger inequality to $G(\cdot,y)$
on the set $B_1(y)$: using~\eqref{eq:choice-0}, we have
$$
\int_{B_1(y)} | G(x,y)|dx 
\leq C 
\int_{B_1(y)} |\nabla_x G (x,y)|dx.
$$
Applying~\eqref{eq:estimee-gradient-G-1} with $\Omega = B_1(y)$,
we deduce that
\begin{equation}
  \label{eq:borne-B1}
  \int_{B_1(y)} | G(x,y)|dx \leq C,
\end{equation}
where $C$ does not depend on $y$. 

We next define, for any $R>0,$ the function
$$
f(R) = \frac 1 {2\pi R} \int_{\partial B_R(y)} | G(x,y)|dx,
$$
where $dx$ denotes the Lebesgue measure on the circle $\partial
B_R(y)$. Note that $f$ depends on $y$, but we keep this dependency
implicit in our notation. In the sequel of the proof, we first show a
bound on $f$ (step 3a), and then deduce from that bound a bound on $G$
(step 3b).  

\medskip 

\noindent{\bf Step 3a: bound on $f$}

We have, for any $R>R'>0$:
\begin{multline}
\label{eq:borne-R-diff}
\left| f(R) - f(R') \right| 
\leq 
\int_{R'}^{R} \left| f'(r) \right| dr 
\leq 
\int_{R'}^{R} \frac{1}{2\pi r} \int_{\partial B_r(y)} 
\left| \nabla_x G(x,y) \right| dx dr 
\\
\leq 
\frac{1}{2\pi R'} \int_{B_R(y) \setminus B_{R'}(y)} |\nabla_x G(x,y)| dx 
\leq
C \frac{\sqrt{R^2 - {R'}^2}}{R'} = C \frac{R}{R'},
\end{multline}
where we have again used~\eqref{eq:estimee-gradient-G-1} and
where the constant $C$ does not depend on $y$. This implies that $f(R)$
is bounded independently of $R$ and $y$ for $R\in (1/2, 1)$. Indeed, for
such an $R$, we rewrite~\eqref{eq:borne-R-diff} as
$f(R) \leq f(R') + C R/R'$ (recall that $f$ is non-negative),
and integrate with respect to $R'$ between $1/4$ and $1/2$, finding
$$
\frac{1}{4} f(R) \leq \int_{1/4}^{1/2} f(R')dR' + CR.
$$
Using~\eqref{eq:borne-B1}, we infer
\begin{equation}
  \label{eq:borne_R_pas_trop_petit}
\forall R\in \left[\frac12,1\right], \quad  f(R) \leq C, 
\end{equation}
for some constant $C$ independent of $R$ and $y$. Next, we consider two
different cases: $R>1$ and $R<1/2$.

\begin{itemize}
\item Case $R>1$: in such a case, we define $p\in \NN$ such that 
$$
\frac12 < \frac{R}{2^p} \leq 1,
$$
that is, $p$ is the integer part of $\frac{\log R}{\log 2}$, which reads
$\frac{\log R}{\log 2} \leq p < \frac{\log R}{\log 2}+1$. We then
apply~\eqref{eq:borne-R-diff} with $R = 2^{-j} R,$ $R' = 2^{-j-1} R$,
finding $$
\left| f \left(\frac R {2^j}\right) - f \left(\frac R {2^{j+1}} \right) \right|
\leq C,
$$
where $C$ is a constant which does not depend on $R$, $j$, nor on
$y$. We sum up all these inequalities for $0 \leq j \leq p-1$, and
obtain  
$$
f(R) \leq f\left(\frac{R}{2^p}\right) + Cp.
$$
Recalling~\eqref{eq:borne_R_pas_trop_petit} and the definition of $p$,
we infer 
\begin{equation}
  \label{eq:borne_R_grand}
  |f(R)| \leq C \left( 1 + |\log(R)|\right),
\end{equation}
where $C$ is independent of $R$ and $y$.

\item Case $R<1/2$: the approach is similar to the preceding case. We
  define $p\in \NN$ such that
$$
\frac12 \leq 2^p R < 1,
$$
that is, $p$ is the integer part of $-\frac{\log R}{\log 2}-1$. We
apply~\eqref{eq:borne-R-diff} with $R'=2^jR$ and $R = 2^{j+1}R$, finding
$$
\left| f(2^j R) - f(2^{j+1}R) \right| \leq C.
$$
We sum this with respect to $0\leq j\leq p-1$, and find
that~\eqref{eq:borne_R_grand} is again valid in this case.
\end{itemize}
Collecting the result of the above two cases, we find that
\begin{equation}
  \label{eq:borne-R-partout}
  \forall R>0, \quad   |f(R)| \leq C \left( 1 + |\log(R)|\right),
\end{equation}
where the constant $C$ does not depend on $R$ nor on $y$.

\medskip 

\noindent{\bf Step 3b: bound on $G$}

We first make use of~\eqref{eq:borne-R-partout} to obtain a bound on the
$L^1$ norm of $G$ in any annulus. For any $\beta \leq \gamma$, we indeed
have 
$$
\| G(\cdot,y) \|_{L^1(B_\gamma(y) \setminus B_\beta(y))} 
= 
2 \pi \int_\beta^\gamma r f(r) dr,
$$
hence, using~\eqref{eq:borne-R-partout}, we obtain
\begin{equation}
\label{eq:merc1}
\| G(\cdot,y) \|_{L^1(B_\gamma(y) \setminus B_\beta(y))} 
\leq
C \int_\beta^\gamma r \left( 1 + |\log(r)| \right) dr.
\end{equation}
Consider now $R \geq 1/2$. Then $3R \geq 2R \geq 1$, and~\eqref{eq:merc1}
implies
\begin{multline}
\label{eq:merc2}
\forall R \geq 1/2, \quad
\| G(\cdot,y) \|_{L^1(B_{3R}(y) \setminus B_{2R}(y))} 
\leq
C \int_{2R}^{3R} r \left( 1 + \log(r) \right) dr
\\
\leq
3 C R^2 \left( 1 + \log(3R) \right)
\leq
\mathcal{C} R^2 \left( 1 + \left| \log(R) \right| \right),
\end{multline}
for some $\mathcal{C}$ independent of $R$ and $y$.
In turn, if $R \leq 1/3$, then~\eqref{eq:merc1} implies
\begin{multline}
\label{eq:merc3}
\forall R \leq 1/3, \quad
\| G(\cdot,y) \|_{L^1(B_{3R}(y) \setminus B_{2R}(y))} 
\leq
C \int_{2R}^{3R} r \left( 1 - \log(r) \right) dr
\\
\leq
3 C R^2 \left( 1 - \log(3R) \right)
\leq
\mathcal{C} R^2 \left( 1 + \left| \log(R) \right| \right),
\end{multline}
for some $\mathcal{C}$ independent of $R$ and $y$.

Next, we recall that,
according to Sobolev imbeddings (see for
instance~\cite[Theorem~7.10]{gilbarg-trudinger}), we have 
$$
\forall p<2, \quad \forall u \in W^{1,p}(\RR^2), \quad
\|u\|_{L^\frac{2p}{2-p}(\RR^2)} \leq C_p \| \nabla u\|_{L^p(\RR^2)}.
$$
We apply this inequality to $u = G(\cdot,y) \chi_R$, where $\chi_R$
is a cut-off function satisfying 
$$
\chi_R \in {\cal D}(\RR^2), \quad |\nabla \chi_R|\leq \frac C R, \quad
\chi_R = 0 \text{ outside } B_{3R}(y), \quad \chi_R = 1 \text{ in
} B_{2R}(y).
$$
We find, for $p=1$, that
\begin{multline}
\| G(\cdot,y) \|_{L^2(B_{2R}(y)\setminus B_R(y))} 
\leq
\| u \|_{L^2(\RR^2)}
\leq 
C \| \nabla u \|_{L^1(\RR^2)}
\\
\leq
C \| \nabla_x G(\cdot,y) \|_{L^1(B_{3R}(y))} + \frac{C}{R} 
\| G(\cdot,y) \|_{L^1(B_{3R}(y)\setminus B_{2R}(y))}.
\label{eq:merc6}
\end{multline}
The first term of the right-hand side is bounded
using~\eqref{eq:estimee-gradient-G-1}, that yields 
\begin{equation}
\label{eq:merc4}
\| \nabla_x G(\cdot,y) \|_{L^1(B_{3R}(y))} \leq CR.
\end{equation}
The second term is bounded using~\eqref{eq:merc1},~\eqref{eq:merc2}
and~\eqref{eq:merc3}. If $R \geq 1/2$ or $R \leq 1/3$, we indeed see
from~\eqref{eq:merc2} and~\eqref{eq:merc3} that
\begin{equation}
\label{eq:merc5}
\frac{C}{R} \| G(\cdot,y) \|_{L^1(B_{3R}(y) \setminus B_{2R}(y))} 
\leq
\mathcal{C} R \left( 1 + \left| \log(R) \right| \right).
\end{equation}
In turn, if $1/3 \leq R \leq 1/2$, then we deduce from~\eqref{eq:merc1}
that
$$
\frac{C}{R} \| G(\cdot,y) \|_{L^1(B_{3R}(y) \setminus B_{2R}(y))} 
\leq
3C \| G(\cdot,y) \|_{L^1(B_{3/2}(y) \setminus B_{2/3}(y))} 
\leq \mathcal{C},
$$
and hence~\eqref{eq:merc5} is again valid. 

Collecting~\eqref{eq:merc6},~\eqref{eq:merc4} and~\eqref{eq:merc5}, we
obtain 
$$
\|  G \|_{L^2(B_{2R}(y)\setminus B_R(y))} 
\leq C R + C R \left| \log(R) \right|.
$$
Finally, we apply~\cite[Theorem 2]{moser} (see
also~\cite[Theorem~8.15]{gilbarg-trudinger}), 
which implies that, for
any $v \in W^{1,1}_{\rm loc}(\RR^2)$ such that $Lv = 0$ in
$B_{4R}(y)\setminus B_{R/2}(y)$, we have 
$$
\sup_{B_{2R}(y)\setminus B_R(y)} v \leq \frac{C}{R} 
\| v \|_{L^2(B_{2R}(y)\setminus B_R(y))}.
$$
Applying this to $ G(\cdot,y)$ and $- G(\cdot,y)$, we find
\begin{equation}
\label{eq:bousin}
\sup_{B_{2R}(y)\setminus B_R(y)} | G(\cdot,y)|\leq C \left(1+\left|
    \log(R) \right| \right).
\end{equation}
The function $G$ hence satisfies~\eqref{eq:G-positive-2d}. This
concludes the proof of the existence of a function $G$ satisfying the
properties claimed in Theorem~\ref{th:existence-unicite-2d}.

\bigskip

To prove the uniqueness of $G$ (up to a constant), we follow the same
argument as in the case $d \geq 3$ (see
Theorem~\ref{th:existence-unicite}). Assume that $G_1$ and $G_2$ are two
solutions. We point out that $H = G_1 - G_2$ satisfies $\dive_x 
\left(A\nabla_x H\right) = 0$ for any $y\in \RR^2$. Fixing $y$, we apply
the corollary of~\cite[Theorem 4]{moser}, 
which implies that, if $H$ is not constant, then 
$\sup\left\{H(x,y), |x-y|=r \right\} - \inf\left\{H(x,y),
  |x-y|=r \right\}$ must grow at least like a positive power of $r$ as $r\to
\infty$. This latter behaviour is in contradiction
with~\eqref{eq:G-positive-2d}. Thus $H=G_1-G_2$ is constant. This
concludes the proof of Theorem~\ref{th:existence-unicite-2d}.
\end{proof}

\begin{remark}
The above proof can be adapted to the case of the Green function $G_R$ on
the bounded domain $B_R$, i.e. the solution
to~\eqref{eq:def-G-faible-R}. 
We hence obtain 
$$
\forall (x,y) \in B_R \times B_R, \quad \left| G_R(x,y)\right| \leq
\mathcal{C}_R  + \mathcal{C} \bigl|\log|x-y|\bigr|,
$$
thus recovering the result of~\cite[Section 6]{dolzmann}.
Note that the constant $\mathcal{C}_R$ in the above bound a priori depends on
$R$. Think indeed for instance of the case $L = -\Delta$, where $G_R(x,0) =
-\log|x| + \log R$.
\end{remark}

\section{Asymptotic behaviour}
\label{sec:asymptotique}

We now give some results about the asymptotic behaviour (at infinity and
at the origin) of the Green function $G$. 
First, we note that, collecting~\eqref{eq:G-positive}
and~\eqref{eq:G-positive-2d}, we have the following: 
\begin{proposition}
Assume that $A$ satisfies~\eqref{eq:A-coercive}
and~\eqref{eq:A-bornee}. Then, the Green function $G$ of the operator
$-\dive\left(A\nabla \cdot\right)$ (namely the solution
to~\eqref{eq:def-G-faible}) satisfies
\begin{equation}
\label{eq:borne-G}
\exists C>0, \ \forall (x,y)\in \RR^d\times \RR^d, \ 
|G(x,y)| \leq
\begin{cases}
C\left( 1 + \left| \log|x-y|\right|\right) & \text{if } d=2, \\
  C|x-y|^{2-d} &\text{if } d>2.
\end{cases}
\end{equation}
\end{proposition}
As we pointed out in the introduction, this result is well-known for
bounded
domains~\cite{avellaneda-lin-87,dolzmann,gruter,kozlov,littmann}.
However, almost all results are limited to this case, except
for~\cite[Theorem 13]{avellaneda-lin-87}, for which "in spirit", the
domain is infinite due to the scaling with respect to $\varepsilon \to
0$. The articles~\cite{kozlov} and~\cite[Section 10]{littmann} also
consider the case of unbounded domains (see also a remark
following~\cite[Theorem~3.3]{gruter}), but do not consider the case $d=2$. 
Finally, the appendix of \cite{kenig-ni} treats the case of $\RR^2$.

\medskip

Next, we give results on the gradient of $G$.
\begin{proposition}
\label{pr:estimee-gradient}
Assume that $A$
satisfies~\eqref{eq:A-periodique},~\eqref{eq:A-holder},~\eqref{eq:A-coercive}
and~\eqref{eq:A-bornee}. Then the Green function $G$ associated with $L
= -\dive\left(A\nabla \cdot\right)$ satisfies the following estimates: 
\begin{gather}
\label{eq:estimee-gradient_x-G}
\exists C>0, \quad \forall x\in \RR^d, \quad \forall y \in \RR^d, \quad
|\nabla_x G(x,y) |\leq \frac{C}{|x-y|^{d-1}},
\\
\label{eq:estimee-gradient_y-G}
\exists C>0, \quad \forall x\in \RR^d, \quad \forall y \in \RR^d, \quad
|\nabla_y G(x,y) |\leq \frac{C}{|x-y|^{d-1}}.
\end{gather}
\end{proposition}

Similar results are given in~\cite[Theorem 3.3]{gruter}, in the case of
bounded domains. 

\begin{proof} 
We start with the case $d\geq 3$, and apply~\cite[Lemma
16]{avellaneda-lin-87} to $G$ as a function of $x$, which implies that
\begin{multline}
\label{eq:avellaneda}
\forall x \in \RR^d, \quad \forall y\in\RR^d, \quad \forall r < |x-y|,
\\ 
\|\nabla_x G(\cdot,y)\|_{L^\infty(B_{r/2}(x))} \leq \frac{C}{r} \,
\|G(\cdot, y)\|_{L^\infty(B_r(x))},
\end{multline}
where $C$ depends only on $\|A\|_{C^{0,\delta}}$, $\delta$, $\alpha$
and $d$. Using~\eqref{eq:borne-G}, we thus obtain
\begin{equation}
\label{eq:titi}
|\nabla_x G(x,y)| \leq \frac{C}{r} \sup_{z \in B_r(x)} 
\frac{1}{|z-y|^{d-2}}.
\end{equation}
Note that we have used $|\nabla_x
G(x,y)|\leq \|\nabla_x G(\cdot,y)\|_{L^\infty(B_{r/2}(x))}$, which is
true only almost everywhere. However, changing the function on a set of
measure zero if necessary, it is possible to assume that this inequality
holds everywhere.
Taking $r = |x-y|/2$, we have, for any $z \in B_r(x)$,
$$
|x-y| \leq |x-z| + |z-y| \leq r + |z-y| = \frac12 |x-y| + |z-y|.
$$
We hence deduce from~\eqref{eq:titi} that
$$
|\nabla_x G(x,y)| \leq \frac{C}{r} \ \left( \frac{2}{|x-y|} \right)^{d-2} 
= \frac{2^{d-1} \, C}{|x-y|^{d-1}}.
$$
This proves~\eqref{eq:estimee-gradient_x-G}.

Next, in order to prove~\eqref{eq:estimee-gradient_y-G}, we point out
that $G^\star(x,y) := G(y,x)$ is the Green function of the operator
$L^\star$ defined by
\begin{equation}
\label{eq:L-star}
L^\star u = -\dive\left(A^T\nabla u\right).
\end{equation}
A proof of this fact\footnote{The main idea of the proof consists in
  choosing the test function $\varphi(x) = G(z,x)$
  in~\eqref{eq:def-G-faible}, for any 
$z \in \RR^d$, and next multiplying~\eqref{eq:def-G-faible} by an
arbitrary function $f(y)$ and integrating over $y$.
However, as the function $G(z,\cdot)$ does not belong to
${\cal D}(\RR^d)$, some regularization arguments are in order.}
is given in~\cite[Theorem 1.3]{gruter}
and~\cite[Theorem 1]{dolzmann} in the case $d \geq 3$, and this proof
carries over to the case $d=2$. 
Hence, applying~\eqref{eq:estimee-gradient_x-G} to $G^\star$, we
deduce~\eqref{eq:estimee-gradient_y-G}.

\bigskip

We turn to the case $d=2$. The estimate~\eqref{eq:avellaneda} is
not sufficient here, since $G(x,y)$ is not bounded as $|x-y|\to \infty$.
Instead, we use the same trick as in the proof of~\cite[Theorem
13]{avellaneda-lin-87}, using~\eqref{eq:estimee-gradient_x-G} for $d=3$.
For this purpose, we introduce the operator $\widetilde L$ defined on
$H^1(\RR^3)$ by 
\begin{equation}
\label{eq:L-tilde}
\widetilde L u = -\dive_x\left(A(x)\nabla_x u\right) - \partial_t^2 u,
\end{equation}
where $x\in\RR^2$ and $t\in \RR$. Let $\widetilde G$ be the associated
Green function. According to the above proof and
to~\eqref{eq:borne-G}, we have 
$$
\left|\widetilde G(x,t,y,s)\right| \leq \frac C {|x-y| + |t-s|},
$$
and
\begin{equation}
\label{eq:grad-G-tilde}
\left|\nabla_x \widetilde G(x,t,y,s)\right| +
\left|\partial_t \widetilde G(x,t,y,s)\right| \leq \frac{C}{|x-y|^2 +
  (t-s)^2}.
\end{equation}

Next, we set, for any $x$ and $y$ in $\RR^2$, with $x \neq y$,
$$
G_\kappa(x,y) = \int_{-\kappa}^\kappa \widetilde G(x,t,y,0) dt.
$$
We deduce from~\eqref{eq:grad-G-tilde} that
\begin{equation}
\label{eq:grad}
\left|\nabla_x G_\kappa(x,y)\right| \leq C\int_{-\infty}^\infty
\frac{dt}{|x-y|^2 + t^2} = \frac{C\pi}{|x-y|},
\end{equation}
for a constant $C$ independent of $\kappa$, $x$ and $y$. Hence,
$\nabla_x G_\kappa$ is 
bounded in $L^p_{\rm loc}(\RR^2 \times \RR^2)$, uniformly with respect
to $\kappa$, for any $p<2$. Thus, for any $R>0$, extracting a
subsequence if necessary, $\nabla_x G_\kappa$ 
converges weakly in $(L^p(B_R \times B_R))^2$ to some $T\in
(L^p(B_R \times B_R))^2$. Now, we have, in the sense of
distribution,
$$
\partial_{x_1} \partial_{x_2} G_\kappa = \partial_{x_2}\partial_{x_1} G_\kappa.
$$
This property passes to the limit, so that $\partial_{x_1} T_2
= \partial_{x_2} T_1$. This implies that $T = \nabla_x
\overline{G}$ for some $\overline{G} \in
W^{1,p}(B_R \times B_R)$. We next point out that the limit
$\overline{G}$ does not depend on $R$, in the sense that if $R'>R$, then
$\overline G'$ defined on $B_{R'}\times B_{R'}$ as above satisfies
$\nabla_x \overline G' = \nabla_x \overline G.$
We thus have $\overline{G} \in L^\infty_{y,\rm
  loc}\left(\RR^2,W_{x,\rm loc}^{1,1} \left(\RR^2\right)\right)$.

Note also that~\eqref{eq:grad} implies that, for any $y \in \RR^2$, the
function $\nabla_x G_\kappa(\cdot,y)$ is
bounded in $L^p_{\rm loc}(\RR^2)$, uniformly with respect
to $\kappa$, for any $p<2$. Thus, for any bounded domain $B_R$, extracting a
subsequence if necessary, $\nabla_x G_\kappa(\cdot,y)$ 
converges weakly in $(L^p(B_R))^2$, and, by uniqueness, $\nabla_x
G_\kappa(\cdot,y)$ converges to $\nabla_x \overline{G}(\cdot,y)$ weakly in
$(L^p(B_R))^2$. 

At this point, $\overline{G}(\cdot,y)$ is only determined up to an
additive constant. We now fix this constant (and hence uniquely defined
$\overline{G}(\cdot,y)$) by assuming that
$$
\int_{B_1(y)} \overline{G}(x,y) \, dx = 0.
$$

\medskip

In the sequel, we show that $\overline{G}$ satisfies all the properties of
Theorem~\ref{th:existence-unicite-2d}. By uniqueness of the Green
function $G$ up to an additive constant, we will obtain that
$\overline{G} = G$ up to a constant. We will then deduce bounds on
$\nabla G$ from the bounds we have on $\nabla \overline{G}$. 

\medskip

We first show that $\overline{G}$ satisfies~\eqref{eq:def-G-faible}.
Consider $\varphi \in {\cal D}(\RR^2)$ and $\psi \in {\cal
  D}(\RR)$. Considering the test function $\psi(t) \varphi(x)$
in~\eqref{eq:def-G-faible}, we see that the Green function $\widetilde
G$ satisfies the weak formulation
\begin{multline*}
\int_{\RR^3} 
\psi(t) \left( \nabla \varphi(x) \right)^T A(x) 
\nabla_x \widetilde G(x,t,y,0) \, dx dt 
\\
+
\int_{\RR^3} 
\varphi(x) \, \psi'(t) \, \partial_t \widetilde G(x,t,y,0) \, dx dt
=
\varphi(y) \psi(0).
\end{multline*}
Consider $\psi$ such that $\psi(t) = 1$ whenever $|t| \leq \kappa$, $\psi(t)
= 0$ whenever $|t| \geq 1+\kappa$, and 
$\max(\| \psi \|_{L^\infty(\RR)},\| \psi' \|_{L^\infty(\RR)}) \leq 1$. We
have
\begin{equation}
\label{eq:weak_R}
\int_{\RR^2} 
\left( \nabla \varphi(x) \right)^T A(x) \nabla_x G_\kappa(x,y) 
\, dx 
+
e_1(\kappa) 
+ 
e_2(\kappa)
=
\varphi(y),
\end{equation}
with
\begin{eqnarray*}
e_1(\kappa) &=& \int_{\RR^2} \int_{\kappa \leq |t| \leq 1+\kappa}
\psi(t) \left( \nabla \varphi(x) \right)^T A(x) 
\nabla_x \widetilde G(x,t,y,0) \, dx dt, 
\\
e_2(\kappa) &=& \int_{\RR^2} \int_{\kappa \leq |t| \leq 1+\kappa}
\varphi(x) \, \psi'(t) \, \partial_t \widetilde G(x,t,y,0) \, dx dt.
\end{eqnarray*} 
Let us now bound from above $e_1$ and $e_2$. 
Using~\eqref{eq:grad-G-tilde}, and introducing a compact $K \subset \RR^2$ containing
the support of $\varphi$, we have
\begin{eqnarray*}
\left| e_1(\kappa) \right| & \leq & \| \psi \|_{L^\infty} 
\| A \|_{L^\infty} \| \nabla \varphi \|_{L^\infty} 
\int_{K} \int_{\kappa \leq |t| \leq 1+\kappa}
\left| \nabla_x \widetilde G(x,t,y,0) \right| \, dx dt
\\
& \leq & \| \psi \|_{L^\infty} 
\| A \|_{L^\infty} \| \nabla \varphi \|_{L^\infty} 
\int_{K} \int_{\kappa \leq |t| \leq 1+\kappa} \frac{C}{|x-y|^2 + t^2} \, dx dt
\\
& \leq & \| \psi \|_{L^\infty} 
\| A \|_{L^\infty} \| \nabla \varphi \|_{L^\infty} 
\ \mu(K) \ \frac{C}{\kappa^2}.
\end{eqnarray*}
Hence, $e_1(\kappa)$ vanishes when $\kappa \to \infty$.
Likewise, $e_2(\kappa)$ also vanishes when $\kappa \to \infty$. Passing to the
limit $\kappa \to \infty$ in~\eqref{eq:weak_R}, and using that $\nabla_x
G_\kappa(\cdot,y)$ weakly converges to $\nabla_x \overline{G}(\cdot,y)$, 
we deduce that, for any 
$\varphi \in {\cal D}(\RR^2)$, we have 
$$
\int_{\RR^2} 
\left( \nabla \varphi(x) \right)^T A(x) \nabla_x \overline{G}(x,y) 
\, dx 
=
\varphi(y).
$$
We have thus obtained that the function $\overline{G} \in L^\infty_{y,\rm
  loc}\left(\RR^2,W_{x,\rm loc}^{1,1} \left(\RR^2\right)\right)$ 
satisfies~\eqref{eq:def-G-faible}. Assume now that $\overline{G}$ also
satisfies~\eqref{eq:G-positive-2d}. Then, according to the uniqueness
of $G$ (see Theorem~\ref{th:existence-unicite-2d}), we have 
$\nabla_x G = \nabla_x \overline{G}$. 

In turn, we deduce from~\eqref{eq:grad} that
\begin{equation}
\label{eq:toto}
\left| \nabla_x G(x,y)\right| 
=
\left| \nabla_x \overline{G}(x,y)\right| \leq \frac{C\pi}{|x-y|}.
\end{equation}
This hence proves the estimate~\eqref{eq:estimee-gradient_x-G} in the
case $d=2$. 

To prove~\eqref{eq:estimee-gradient_y-G} in the case $d=2$, we again use
the fact that $G(y,x)$ is the Green function of $L^\star$ defined
by~\eqref{eq:L-star}, so the estimate~\eqref{eq:estimee-gradient_x-G}
that we have just shown implies~\eqref{eq:estimee-gradient_y-G}. 

\medskip

There only remains to prove that $\overline G$ satisfies
\eqref{eq:G-positive-2d}. To this end, we note that~\eqref{eq:toto}
implies the estimate \eqref{eq:estimee-gradient-G-1}, for $\Omega$ a
ball or an annulus of the form $B_{2R} \setminus B_R$. Hence, the end of
the proof of Theorem~\ref{th:existence-unicite-2d} applies here, leading
from~\eqref{eq:estimee-gradient-G-1} to~\eqref{eq:bousin}, which implies
that $\overline G$ satisfies~\eqref{eq:G-positive-2d}.
\end{proof}

\begin{remark}
The above arguments indicate two different proofs for the existence of $G$
in dimension two: the first one consists in defining the Green function
on the bounded domain $B_R$, and then letting $R\to\infty$, as it is
done in the proof of Theorem~\ref{th:existence-unicite-2d}. The
second strategy uses the three-dimensional Green function
$\widetilde{G}$ of the operator $\widetilde{L}$ defined by
\eqref{eq:L-tilde}. One 
integrates $\widetilde{G}$ with respect to the third variable, finding a
Green function for the operator $L$ in dimension two. This approach is
used in the proof of Proposition~\ref{pr:estimee-gradient}.

Note also that Proposition~\ref{pr:estimee-gradient} is proved under
stronger assumptions than Theorem~\ref{th:existence-unicite-2d}.
\end{remark}

We next prove upper bounds on $\nabla_x \nabla_y G$.
\begin{proposition}
\label{pr:estimee-gradient-croise}
Assume that $A$
satisfies~\eqref{eq:A-periodique},~\eqref{eq:A-holder},~\eqref{eq:A-coercive}
and~\eqref{eq:A-bornee}. Then the Green function $G$ associated with $L
= -\dive\left(A \nabla \cdot\right)$ satisfies the following estimate:
\begin{equation}
\label{eq:estimee-gradient-croise}
\exists C>0, \quad \forall x\in \RR^d, \quad \forall y \in \RR^d, \quad
|\nabla_x \nabla_y G(x,y) |\leq \frac{C}{|x-y|^d}.
\end{equation}
\end{proposition}

Here again, similar results for the Green function in a {\em bounded}
domain are given in the literature, for instance in~\cite[Theorem
3.3]{gruter}.

\begin{proof} 
We have, in the sense of distribution,
$$
-\dive_x\left(A(x)\nabla_x \nabla_y G(x,y)\right) = 0 \quad 
\text{in $B_\delta(y)^C$, for any $\delta > 0$}.
$$
We can thus apply~\cite[Lemma 16]{avellaneda-lin-87}, and obtain, as
in~\eqref{eq:avellaneda}, that
\begin{multline*}
\forall x \in \RR^d, \quad \forall y\in\RR^d, \quad \forall r < |x-y|,
\\ 
\|\nabla_x \nabla_y G(\cdot,y)\|_{L^\infty(B_{r/2}(x))} \leq \frac{C}{r} \,
\|\nabla_y G(\cdot, y)\|_{L^\infty(B_r(x))}.
\end{multline*}
Using~\eqref{eq:estimee-gradient_y-G}, we
deduce~\eqref{eq:estimee-gradient-croise}. 
\end{proof}

Using arguments similar to those used to prove
Propositions~\ref{pr:estimee-gradient}
and~\ref{pr:estimee-gradient-croise}, we also show the following result 
on the Green function $G_R$ of the operator $-\dive\left(A\nabla
  \cdot\right)$ on the bounded domain $B_R$ with homogeneous Dirichlet
boundary conditions. The interest of this result is the
independence of the obtained bounds with respect to the size of the
domain $B_R$. 

\begin{proposition}
\label{pr:estimee-gradient-R}
Assume that $A$
satisfies~\eqref{eq:A-periodique},~\eqref{eq:A-holder},~\eqref{eq:A-coercive}
and~\eqref{eq:A-bornee}. Let 
$G_R$ be the Green function of the operator $-\dive\left(A\nabla
  \cdot\right)$ on $B_R$ with homogeneous Dirichlet boundary conditions
(namely, $G_R$ is the unique solution to~\eqref{eq:def-G-faible-R} with
the boundary condition $G_R(x,y) = 0$ if $|x|=R$).

Then, there exists a constant $C$ such that, for any $R>0$, 
\begin{eqnarray}
\label{eq:estimee-gradient_x-GR}
\forall (x,y) \in B_R \times B_R, \quad
|\nabla_x G_R(x,y) | \leq \frac{C}{|x-y|^{d-1}},
\\
\label{eq:estimee-gradient_y-GR}
\forall (x,y) \in B_R \times B_R, \quad
|\nabla_y G_R(x,y) | \leq \frac{C}{|x-y|^{d-1}},
\\
\label{eq:estimee-gradient_xy-GR}
\forall (x,y) \in B_R \times B_R, \quad
|\nabla_x\nabla_y G_R(x,y) | \leq \frac{C}{|x-y|^{d}}.
\end{eqnarray}
\end{proposition}

\section{Extensions}
\label{sec:extensions}

First, it should be noted that, assuming further regularity on the
coefficients of the matrix $A$, it is possible to prove more precise
decaying properties of the Green function. This was proved
in~\cite{sevo}.

Next, it is clearly possible to adapt the technique of~\cite{dolzmann}
and~\cite{fuchs} (see also~\cite{hildebrandt-1,hildebrandt-2}) to treat
the case of {\em systems} of elliptic PDEs. This case is also considered
in~\cite{avellaneda-lin-87}. 

Another question is the extension of the present results to the case of
non-periodic coefficients. This is, for instance, what is done
in~\cite{gruter} and~\cite{dolzmann}, in the case of bounded
domains. However, some of the estimates we have used here (in
particular~\eqref{eq:avellaneda}) do rely on the fact that the matrix is
periodic. Thus, the extension is not straightforward.

Finally, it should be possible to extend our results to the case of
piecewise H\"older coefficients. For instance, gradient estimates for
elliptic equations with such discontinuous coefficients are derived
in~\cite{li-vogelius}. It is probably possible to use them in the
setting of the current article, but such a work remains to be done.

\section*{Acknowledgments} 

FL gratefully acknowledges the support from EOARD under Grant
FA8655-10-C-4002 and the support from ONR under Grant 
N00014-09-1-0470.

\end{document}